\documentclass[12pt, a4paper]{amsart}

\usepackage[euler-digits]{eulervm}
\usepackage{graphicx}
\usepackage[all,cmtip]{xy}
\usepackage{comment}
\usepackage{amsmath}
\usepackage[dvipsnames]{xcolor}

\usepackage{tikz-cd}
\usepackage{tikz}

\usepackage[pagebackref]{hyperref}
\hypersetup{colorlinks=true, linkcolor=Red, citecolor=NavyBlue}

\usepackage{amsfonts, amsthm, amssymb, amsmath}
\usepackage{mathrsfs,array}
\usepackage{eucal,times,color,enumerate,accents}
\usepackage{url}
\usepackage{float}
\usepackage{pbox}
\usepackage[headings]{fullpage}
\usepackage{fullpage}
\usepackage{microtype}
\usepackage{ytableau}
\usepackage{color}
\usepackage{mathrsfs}
\usepackage{amssymb}

\usepackage{bm}
\usepackage{enumerate}

\usepackage{amssymb}
\usepackage{ulem}

\newcommand{\bbC}{{\mathbb C}}
\newcommand{\bbZ}{{\mathbb Z}}

\newcommand{\bbP}{{\mathbb P}}

\newcommand{\Oh}{{\mathcal O}}

\DeclareMathOperator{\HH}{H}

\DeclareMathOperator{\Ext}{Ext}

\newcommand{\onto}{\twoheadrightarrow}

\newcommand{\into}{\hookrightarrow}

\newcommand{\tensor}{\otimes}
\newcommand{\isom}{\cong}

\theoremstyle{plain}
\newtheorem{lemma}{Lemma}
\newtheorem{thm}{Theorem}

\newtheorem{cor}{Corollary}
\newtheorem{prop}{Proposition}
\newtheorem{claim}{Claim}
\theoremstyle{definition}

\newtheorem{defn}{Definition}

\newtheorem{remark}{Remark}

\long\def\comment#1{}
\begin{document}

\title{Extendable codimension two subvarieties in a general hypersurface}
\author{G.~V.~Ravindra} 
\address{Department of Mathematics, 
University of Missouri -- St. Louis, St. Louis, MO 63121, USA}
\email{girivarur@umsl.edu} 
\author{Debaditya Raychaudhury}
\address{Department of Mathematics, University of Arizona, Tucson, AZ 85721, USA}
\email{draychaudhury@arizona.edu}

\begin{abstract}
We exhibit a class of {\it extendable} codimension $2$ subvarieties in a general hypersurface of dimension at least 
{\color{black} 4} in projective space. As a consequence, we prove 
that a general hypersurface of degree $d$ {\color{black} and dimension at 
least $4$}
does not support globally generated indecomposable ACM bundles of any rank if their first Chern class $e \ll d$.
\end{abstract}

\date{\today}
\maketitle

\pagenumbering{Alph}

\pagenumbering{arabic}

\thispagestyle{empty}

\section{Introduction}
Let $Y$ be a smooth projective variety and $X\subset Y$ be a smooth subvariety. Relating the geometry
of $X$ and $Y$ has been a long standing theme in algebraic geometry. Results in this context are usually
referred to as Lefschetz theorems. The best known results are the {\it Grothendieck-Lefschetz} and
 {\it Noether-Lefschetz} theorems.  A special case of the Noether-Lefschetz theorem says that for a
 very general hypersurface $X\subset \bbP^3$ of degree $d\geq 4$, any curve $C\subset X$ is a complete
 intersection in $\bbP^3$. In particular, $C=X\cap{S}$ for a surface  $S\subset\bbP^3$ and thus
 {\it extendable} (there is a related notion of extendability in the literature, a very 
 nice survey on which can be found in \cite{Lop}. See the references therein, especially \cite{Wahl87} and \cite{BM87}).  

 \smallskip
 
 More generally, throughout this article, we will say a codimension $k$ subscheme $Z\subset X$ of a smooth 
 hypersurface $X\subset\mathbb{P}^{n+1}$ is {\it extendable} if $Z=X\cap \Sigma$ {scheme-theoretically,} where 
 $\Sigma\subset \mathbb{P}^{n+1}$ is a {pure} codimension $k$ subscheme. 

\smallskip
 
 With a view to finding a generalisation of the Noether-Lefschetz theorem, Griffiths and Harris in \cite{GH}, 
 asked whether any curve in a general hypersurface  $X\subset \bbP^4$ of degree $d\geq 6$ is extendable. 
 The main idea is that codimension $2$ subvarieties in projective spaces are already more complicated 
 (for instance, not all of them are defined by $2$ homogeneous polynomials) and the expectation was that 
 perhaps the codimension $2$ geometry of general hypersurfaces is no more complicated, thus establishing a
 Lefschetz type result.

 \smallskip
 
 C.~Voisin in \cite{Voisin} showed the existence of curves in smooth hypersurfaces $X \subset \bbP^4$ 
 which were not cut out by 
 surfaces in $\bbP^4$. One of the fundamental differences in these two cases is the following. Consider the normal 
 bundle sequence for the inclusions $C\subset X \subset \bbP^{n+1}$:
 $$0 \longrightarrow N_{C/X} \longrightarrow N_{C/\bbP^{n+1}} \longrightarrow \Oh_C(d) \longrightarrow 0.$$
 For smooth hypersurfaces in $\bbP^3$, this sequence splits if and only if $C$ is extendable and hence
 a complete intersection
 (see \cite{GH1}). However, this is no longer true once $C$ is a curve in a smooth hypersurface $X\subset \bbP^4$.
In this case, the splitting of the above sequence only implies that $C$ is {\it infinitesimally extendable}, i.e., there
exists a curve $D\subset X_{(1)}$ where $X_{(1)}$ is the first order thickening of $X$ in $\bbP^4$ such that 
$C=D\cap{X}$. If $C\subset X$ (or more generally a codimension $2$ subvariety $Z$ in a smooth hypersurface
of dimension $n\geq 4$) is, in addition, {\it arithmetically Cohen Macaulay} (henceforth, we abbreviate this as {\it 
ACM}), then it was shown in \cite{MPRV} that if $C$ extends
infinitesimally, then it is in fact extendable in the above sense (see Corollary \ref{mprvcor} for the precise statement). This fact was used to show the existence of a large 
class of counterexamples generalising
Voisin's examples in \cite{Voisin}.  There are also examples of non-extendable subvarieties in higher codimension (see 
for instance \cite{IN}).

\smallskip
 
 Coming back to the case of curves in hypersurfaces in $\bbP^4$, and their extendability, a conjecture in \cite{RT4}
 proposes that any ACM curve $C$ in a general hypersurface $X\subset \bbP^4$ of degree $d\geq 6$ 
  is extendable if the number of generators of the canonical module of the curve $C$ 
 is less than or equal to $2$. When the canonical module has a single generator, the curve $C$ is {\it subcanonical}
 and  the main  result of \cite{R} (see also \cite{MPR2}) states that $C$ is in fact a complete intersection. When the number 
 of generators of the canonical module is $2$, barring  a few exceptions, this conjecture was settled in \cite{RT6}.

 \smallskip

 Extendability of codimension $2$ ACM subvarieties in smooth hypersurfaces is related to  a conjecture of 
 Buchweitz-Greuel-Schreyer (\cite{BGS}) on the non-existence of low rank indecomposable ACM vector bundles 
 and a generalisation of this conjecture (see \cite{Faenzi} and \cite{RT4}), results on which are proven, 
 for example, in \cite{AT1, AT2,RT4,RT6}. It is also related to the {\it Ulrich complexity} of hypersurfaces 
 (\cite{Bea',ES}); we refer to \cite{EC17, Bea, CMRPL} for an overview of this topic, see also \cite{RT6,LR1,LR2}.

 \smallskip
 
 In this article, we exhibit a bigger class of extendable surfaces $Z$ in a general hypersurface 
 $X \subset {\color{black} \bbP^5}$  of degree $d$. As a consequence, we prove a splitting result for ACM bundles 
 $E$ on $X$. The expert will immediately see that the results in this article are far from being sharp.
 Indeed our aim here has been to showcase how well-known and beautiful results available in the 
 literature can be brought together to answer some rather long standing questions of interest.

 \medskip

\noindent{\it Conventions.} We work over the field of complex numbers $\mathbb{C}$. A {\it variety} is an integral separated scheme of finite type over $\mathbb{C}$. A {\it curve} (resp. {\it surface}) is a variety of dimension one (resp. two). 

\medskip

\noindent{\it Acknowledgements.} The authors are grateful to Amit Tripathi for very useful discussions. We are also grateful to the anonymous referee for a number of corrections and invaluable suggestions. The first author acknowledges 
support from the Simons foundation. The second author is partially supported by an AMS-Simons Travel Grant.
		
\section{Statements of the main results}\label{2}
In this section, we provide the statements of our main results. We start by recalling that a subvariety $Z\subset W$ is said {\it ACM} if $H^i_*(W,I_{Z/W})=0$ for $1\leq i\leq \dim Z$ where $I_{Z/W}$ is the ideal sheaf of $Z$ in $W$. 

{Given any coherent sheaf $\mathcal{F}$ on an ACM
variety $Y \subset \bbP^N$, note that the 
module of global sections $\Gamma(\mathcal{F}):=\bigoplus_{a\in\bbZ}\HH^0(X, \mathcal{F}(a))$ is a 
finitely generated 
module over the polynomial ring in $N+1$-variables. Any choice of a  set of {\it generators} $\{s_i\}$ with 
${\rm degree}(s_i)=a_i$, yields a surjection
$$\bigoplus_i \Oh_Y(-a_i) \onto \mathcal{F}$$
which induces a surjection at the level of the $\Gamma$ modules, i.e., a surjection 
$$\bigoplus_i\HH^0(Y,\mathcal{O}_{Y}(a-a_i))\longrightarrow \HH^0(Y,\mathcal{F}(a))\textrm{ for all } 
a\in\mathbb{Z}.$$
\begin{defn}
   We say that the cohomology of a coherent sheaf $\mathcal{F}$ on an ACM variety $Y\subset\mathbb{P}^N$ is {\it generated in degree $k$} (by $m$ sections) if there exists a set of generators $\{s_i\}$ with  ${\rm degree}(s_i)=k$ (consisting of $m$ elements).
\end{defn}
\begin{remark}
    Note that if the cohomology of $\mathcal{F}$ is generated in degree $k$, then $\mathcal{F}(k)$ is globally generated.
\end{remark}
}

The aim of this article is to prove the following:

\begin{thm}\label{mainthm1}
Let $X\subset \bbP^{n+1}$ be a general hypersurface of dimension {\color{black} $n\geq 4$} and degree $d$. 
A local complete intersection, ACM codimension $2$ subvariety $Z\subset X$ 
is extendable if there exists a positive integer $e$ such that
\begin{enumerate}
\item[(i)] $\binom{e+5}{4} \leq 2d-4$, 
\item[(ii)] $I_{Z/X}(e)$ is globally generated, and 
\item[(iii)] { the cohomology of the line bundle 
$\omega_Z\otimes\omega_X^{-1}$ is generated in 
degree $-e$.}
\end{enumerate}
{\color{black} The same conclusion holds when $\dim{X}=3$ if condition $(ii)$
above is replaced by the stronger condition that ${I}_{Z/\bbP^4}(e)$
is globally generated.}
\end{thm}

\begin{remark}
Note that if $N_{Z/X}$ is the normal bundle of $Z\subset X$ in the above, then $$\omega_Z\otimes\omega_X^{-1}=\det{N_{Z/X}}.$$
\end{remark}

Here's an example of a situation in which such codimension two subvarieties arise. Let $E$ be a globally generated ACM bundle of rank $r$ on a smooth,
degree $d$ hypersurface $X\subset\mathbb{P}^{n+1}$ with  \textcolor{black}{$n\geq 4$}. Any choice of $r-1$ general sections yields an exact sequence
$$0 \longrightarrow \Oh_X^{\oplus r-1} \longrightarrow E \longrightarrow I_{Z/X}(e) \longrightarrow 0.$$
Here $Z\subset X$ is {either empty, or a} codimension 2 subvariety defined by the vanishing of these $r-1$ sections, $I_{Z/X}$ is its ideal sheaf and $e$ is the first
Chern class of $E$. If $e$ satisfies the inequality $(i)$ in Theorem \ref{mainthm2}, then $Z${, if non-empty,} is extendable; i.e., $Z =X\cap{\Sigma}$ for some
pure codimension 2 subscheme $\Sigma\subset \bbP^{n+1}$. {This is because in this case (ii) is obvious as $E$ is assumed to be globally generated, and (iii) can be seen by dualizing the above exact sequence and passing to cohomology (the proof is similar to that of Proposition \ref{serre}).}

\smallskip

The proof of Theorem \ref{mainthm2} is based on an induction argument, the main step of which is proving the assertion when {\color{black} $n=4$}. The main ingredient of the proof in this case is the following:
{\color{black} 
\begin{thm}\label{mainthm2}
Let $X$ be a general hypersurface in $\bbP^5$ of degree $d$ and let $Z 
\subset X$ be an ACM local complete intersection surface. Suppose that $Z$ 
satisfies the following conditions for some integer $e>0$:
\begin{itemize}
\item[(i)] $\binom{e+5}{4} \leq 2d-4$, 
\item[(ii)] there exists a smooth member $\Theta\in |I_{Z/X}(e)|$, 
\item[(iii)] {$I_{Z/X}(e+1)$ is globally generated, and
\item[(iv)] the cohomology of $\omega_Z\otimes\omega_X^{-1}$ is generated in degree $-e$.}
\end{itemize}
For a general hyperplane $H \isom \bbP^4 \subset \bbP^5$, let 
$ C:= Z\cap{H}$, $Y:= X \cap{H}$ be the hyperplane sections. Then the normal bundle sequence 
\begin{equation}\label{nbs}
0 \longrightarrow N_{C/Y} \longrightarrow N_{C/\bbP^4} \longrightarrow \Oh_C(d) \longrightarrow 0.
\end{equation}
associated to the inclusions $ C\subset Y \subset \bbP^4$ splits. 
\end{thm}
}

When $X$ is a general hypersurface in $\bbP^4$ of degree $d$ and $C \subset 
X$ is an ACM curve, we show that the splitting of \eqref{nbs} implies 
extendability of $C$. Similar ideas were used by Voisin in \cite{V} in the 
context of the extendability of curves in K3 surfaces (in the sense 
discussed in \cite{Lop}).

\smallskip

As mentioned in the Introduction, extendability of pure codimension 2 subvarieties is intimately related with the splitting of
ACM vector bundles (cf. Lemma \ref{l3} and Corollary \ref{mprvcor}). In this direction, we deduce the following by-product of our results:
\begin{thm}\label{ggextendable}
Fix a positive integer $e$. Then a general hypersurface of dimension 
{\color{black} $n \geq 4$} and degree $d$ satisfying the inequality 
$\binom{e+5}{4} \leq 2d-4$ does not support any globally generated ACM bundle $E$ with first Chern class $c_1(E) = 
\Oh_X(e)$ {that is not a direct sum of line bundles}.
\end{thm}

Our proof of Theorem \ref{mainthm1} makes use of the Beauville-M\'erindol criterion (see \cite{BM87}) for splitting of short exact sequences, combining it with
Green's exactness criterion for Koszul complexes (see \cite{Green}).

\section{Preliminaries on Hartshorne-Serre correspondence} 
We recall the Hartshorne-Serre correspondence for codimension $2$ subschemes in a smooth variety that will be crucial for us the sequel:
\begin{thm}[{\cite[Theorem 1]{arrondo}}]\label{arr}
Let $X$ be a smooth, projective variety and $Z \subset X$ be a locally complete intersection subvariety of codimension $2$.
Let $L$ be a line bundle such that 
\begin{enumerate}
\item[(i)] $\HH^2(X, L^{-1})=0$, and
\item[(ii)] $\omega_Z\tensor (\omega_X\tensor L)^{-1}$ is globally generated by $(r-1)$ sections.
\end{enumerate}
Then there exists a rank $r$ vector bundle $E$ and an exact sequence
\begin{equation}\label{same}
    0 \longrightarrow \Oh_X^{\oplus r-1} \longrightarrow E \longrightarrow I_{Z/X}\tensor{L} \longrightarrow 0.
\end{equation}
Furthermore, if $\HH^1(X, L^{-1})=0$, then $E$ is unique up to an unique isomorphism.
\end{thm}

\begin{remark}\label{cmpt}
{If $Z\subset X$ satisfies the assumptions of Theorem \ref{arr}, then by definition there is a surjection $$\mathcal{O}_{X}^{\oplus r-1}\longrightarrow \omega_Z\otimes (\omega_X\otimes L)^{-1}.$$ This is the same map that we obtain as connecting map by dualizing the exact sequence \eqref{same}.}
\end{remark}

\begin{remark} {We note the following that will be used without any further reference:
\begin{itemize}
    \item $Y\subset\mathbb{P}^N$ is ACM if and only if it is projectively normal and $H^i_*(Y,\mathcal{O}_Y)=0$ for all $1\leq i\leq \dim Y-1$.
    \item If $Y\subset\mathbb{P}^N$ is ACM and $X\in|\mathcal{O}_Y(d)|$ then $X\subset\mathbb{P}^N$ is ACM.
    \item Let $Y\subset\mathbb{P}^N$ be an ACM variety and $X\in|\mathcal{O}_Y(d)|$. Let $Z\subset X$ be a subvariety of codimension $2$. By the exact sequence
$$0\longrightarrow\Oh_{Y}(-d)\longrightarrow I_{Z/Y}\longrightarrow I_{Z/X}\longrightarrow 0,$$ we see that $Z\subset X$ is ACM if and only if $Z\subset Y$ is ACM.
\end{itemize}}
\end{remark}

{Recall that we say $Y\subset\mathbb{P}^N$ is {\it AG} (i.e. {\it arithmetically Gorenstein}) if it is ACM and subcanonical (i.e. $\omega_Y=\Oh_Y(s)$ for some $s\in\mathbb{Z}$). Let us now record the following useful}

\begin{prop}\label{serre}
{Let $Y\subset\mathbb{P}^N$ be a smooth AG variety and let $X \in \vert \Oh_Y(d) \vert$ be a smooth 
hypersurface and  assume that $\dim X=n\geq 3$. Let $e\in\mathbb{Z}$ and let $Z\subset X$ be an ACM local 
complete intersection subvariety of codimension $2$ for which the cohomology of $\omega_Z\otimes\omega_X^{-1}$ is 
generated in degree $-e$ by $(r-1)$ sections.} Then the associated vector bundle $E$ (coming from Theorem 
\ref{arr}) sitting in the exact sequence
\begin{equation}\label{defeq}
0 \longrightarrow \Oh_X^{\oplus r-1} \longrightarrow E \longrightarrow I_{Z/X}(e) \longrightarrow 0.
\end{equation}
is ACM. Moreover, $E$ is globally generated if and only if $I_{Z/X}(e)$ is globally generated.
\end{prop}
\begin{proof}
Taking dual of \eqref{defeq} gives rise to the 4-term exact sequence
\begin{equation}\label{dualeq}
0 \longrightarrow \Oh_X(-e) \longrightarrow E^\vee \longrightarrow \Oh_X^{\oplus r-1} \longrightarrow \mathcal{E}xt^1_X(I_{Z/X}(e), \Oh_X) \longrightarrow 0.
\end{equation}
One has the identification $\mathcal{E}xt^1_X(I_{Z/X}, \omega_X) \isom \omega_Z$ using which  \eqref{dualeq} may be rewritten as 
\begin{equation}\label{dualeq1}
0 \longrightarrow \Oh_X(-e) \longrightarrow E^\vee \longrightarrow \Oh_X^{\oplus r-1} \longrightarrow \ell \longrightarrow 0
\end{equation}
where $\ell:=\omega_Z\tensor{\omega_X^{-1}}(-e)$. 
Also, by assumption, we have
\begin{equation}\label{surj}
\HH^0(X,\mathcal{O}_X(a)^{\oplus r-1})\longrightarrow \HH^0(Z,\ell(a))\textrm{ surjects for all } a\in\mathbb{Z}
\end{equation}
where the map above is induced by the map $\Oh_X^{\oplus r-1} \longrightarrow \ell$ in \eqref{dualeq1} (see Remark \ref{cmpt}).

Let
$E_1$ be the torsion-free sheaf defined as the
cokernel of the injection $\Oh_X(-e) \longrightarrow E^\vee$ in \eqref{dualeq1}. 
Breaking up the sequence \eqref{dualeq1}, we obtain the two short exact sequences
\begin{equation}\label{ses1}
0 \longrightarrow \Oh_X(-e) \longrightarrow E^\vee \longrightarrow E_1 \longrightarrow 0,
\end{equation}
\begin{equation}\label{ses2}
0 \longrightarrow E_1 \longrightarrow \Oh_X^{\oplus r-1} \longrightarrow \ell \longrightarrow 0.
\end{equation}
Recall that $\HH^i_*(X,\mathcal{O}_X)=0$ for $1\leq i\leq n-1$ as $X\subset\mathbb{P}^N$ is ACM (whence AG by adjunction). Passing to the cohomology of \eqref{ses2}, we conclude that $\HH^1_*(X,E_1)=0$ by \eqref{surj}. Consequently $\HH^1_*(X,E^{\vee})=0$ by \eqref{ses1} which by duality implies $\HH^{n-1}_*(X,E)=0$. It follows that $E$ is ACM since $\HH^i_*(X,E)=0$ for $1\leq i\leq n-2$ by \eqref{defeq}. To see the second assertion, consider the commutative diagram:
\[
\begin{tikzcd}
    0\arrow[r] & \HH^0(X,\Oh_X^{\oplus r-1})\otimes \Oh_X\arrow[r]\arrow[d] & H^0(X,E)\otimes \Oh_X\arrow[r]\arrow[d] & \HH^0(X,I_{Z/X}(e))\otimes\Oh_X\arrow[r]\arrow[d] & 0\\
    0\arrow[r] & \Oh_X^{\oplus r-1}\arrow[r] & E\arrow[r] & I_{Z/X}(e)\arrow[r] & 0.
\end{tikzcd}
\]
Since the left vertical map is surjective, it follows that the middle one is surjective if and only if the right one is so, whence the conclusion follows.
\end{proof}

As an useful consequence, we  deduce the following:
\begin{cor}\label{2gen}
Let the hypotheses be as in Proposition \ref{serre}. Then the multiplication map
$$\HH^0(Z, \ell(a))\tensor\HH^0(Z, \Oh_Z(b)) \longrightarrow \HH^0(Z, \ell(a+b))$$
is surjective whenever $a, b \geq 0$.
\end{cor}
\begin{proof}
Thanks to \eqref{surj} (and the fact that $\HH^0(X,\Oh_X(m))\longrightarrow\HH^0(Z,\Oh_Z(m))$ is surjective for all $m$), it is enough to check that 
$$\HH^0(Z, \Oh_Z(a))\tensor\HH^0(Z, \Oh_Z(b)) \longrightarrow \HH^0(Z, \Oh_Z(a+b))$$
is surjective whenever $a, b \geq 0$. For this, we note that we have a commutative
diagram:
\begin{equation}\label{multiplicationmap}
\begin{tikzcd}
\HH^0(Y, \Oh_{Y}(a))\tensor\HH^0(Y, \Oh_{Y}(b)) \arrow[r]\arrow[d]  & \HH^0(Y, \Oh_{Y}(a+b))\arrow[d] \\
\HH^0(Z, \Oh_Z(a))\tensor\HH^0(Z, \Oh_Z(b))\arrow[r] & \HH^0(Z, \Oh_Z(a+b))
\end{tikzcd}
\end{equation}
The horizontal map on the top row is a surjection as $Y\subset\mathbb{P}^N$ is ACM (hence projectively normal), and the vertical maps are surjective since $Z\subset Y$ is ACM.
It follows that the bottom horizontal map is also a surjection.
\end{proof}

In what follows, we use the results of this section when $Y=\mathbb{P}^{n+1}$, $n\geq 3$, $X\subset\mathbb{P}^{n+1}$ is a smooth hypersurface and $L=\mathcal{O}_X(e)$. Note that in this case, we have 
\begin{equation*}
    \ell:=\omega_Z\otimes(\omega_X\otimes L)^{-1}=\omega_Z(n+2-d-e)
\end{equation*} by adjunction.

\section{Equivalent characterizations of extendability}

In this section, we prove one of the central results that we use to prove our three main theorems. This result and the following corollary are probably well-known to the experts; they had been implicitly used in various articles of the first author, but had not been stated in this explicit form before.

\begin{lemma}\label{l3}
Let $X\subset\mathbb{P}^{n+1}$ be a smooth hypersurface of degree $d$, and let $Z\subset 
X$ be a codimension 2 {local complete intersection} subvariety defined by the exact sequence $$0\longrightarrow
\Oh_X^{\oplus r-1}\longrightarrow E\longrightarrow I_{Z/X}(e)\longrightarrow 0$$ where $E$ is a bundle of rank $r$.
\begin{enumerate}
    \item If $E$ is a direct sum of line bundles, then $Z$ is extendable.
    \item If $Z$ is extendable then the normal bundle sequence 
    \begin{equation}\label{nbsa}
        0\longrightarrow N_{Z/X}\longrightarrow N_{Z/\mathbb{P}^{n+1}}\longrightarrow\Oh_Z(d)\longrightarrow 0
    \end{equation} splits.
\item If the normal bundle sequence \eqref{nbsa} for the inclusions $Z \subset X \subset \bbP^{n+1}$  splits, then there exists a subscheme $Z_{(1)}\supset Z$ in the first-order thickening $X_{(1)}$ of the hypersurface $X$ in
$\bbP^{n+1}$ such that the following sequence is exact 
$$I_{Z_{(1)}/X_{(1)}}(-d)\xrightarrow[]{\times f} I_{Z_{(1)}/X_{(1)}}\longrightarrow I_{Z/X}\longrightarrow 0.$$ 
Furthermore $fI_{Z_{(1)}/X_{(1)}}(-d)= I_{Z/X}(-d)$.
\end{enumerate}    
\end{lemma}
\begin{proof}
(1) Since $E$ splits into a sum of line bundles, the
map $$\Oh_X^{\oplus r-1} \longrightarrow E\isom \bigoplus_{i=1}^r \Oh_X(a_i)$$ lifts to a map 
$$\Oh_{\bbP^{n+1}}^{\oplus r-1} \longrightarrow  \bigoplus_{i=1}^r \Oh_{\bbP^{n+1}}(a_i).$$
The cokernel of this map is (a twist of) the ideal sheaf of a codimension 2 subscheme $\Sigma\subset\mathbb{P}^{n+1}$ which satisfies the condition that $Z=X \cap {\Sigma}$ scheme-theoretically. {Indeed, arguing locally, suppose $X=\textrm{Spec}(A/(f))$ and let the defining ideals of $\Sigma\subset\mathbb{P}^{n+1}$ and $Z\subset\mathbb{P}^{n+1}$ be $J\subset A$ and $I\subset A$ respectively. Now, $(f) \subset I$ as $Z\subset X$. Moreover, $I_{\Sigma/\mathbb{P}^{n+1}}$ restricts to $I_{Z/X}$ which implies $\pi(I)=\pi(J)$ where $\pi:A\longrightarrow A/(f)$ is the natural map. This implies $I=J+ (f)$, or equivalently, $Z=X\cap\Sigma$.} This shows in particular that $\Sigma$ doesn't have a divisorial component. Since $\textrm{codim}_{\mathbb{P}^{n+1}}(\Sigma)\leq 2$ (see for e.g. \cite[Lemma 2.7]{Ott}), we conclude that $\Sigma$ is of pure codimension 2, whence $Z$ is extendable.

(2) Since $Z$ is extendable, $Z=\Sigma\,\cap\,X$ for some pure codimension $2$ subscheme $\Sigma$ in $\bbP^{n+1}$. Hence, $Z \in \vert\Oh_\Sigma(d)\vert$, and so
we have an inclusion 
$$\Oh_Z(d) \isom N_{Z/\Sigma} \into N_{Z/\bbP^{n+1}}.$$ 
 We claim that this inclusion composed with the
surjection 
$$N_{Z/\bbP^{n+1}} \onto N_{X/\bbP^{n+1}\vert Z} \isom \Oh_Z(d)$$ yields a splitting of the 
normal bundle $N_{Z/\bbP^{n+1}}$ and hence the normal bundle sequence \eqref{nbsa} splits.
To verify the claim, we start with the diagram:
\begin{equation*}
    \begin{tikzcd}
    &                                             & 0 \arrow[d] & & \\
                    &                                             & {I}_{X/\bbP^{n+1}} \isom\Oh_{\bbP^{n+1}}(-d) \arrow[d] & & \\
        0 \arrow[r] & {I}_{\Sigma/\bbP^{n+1}} \arrow[r] & {I}_{Z/\bbP^{n+1}} \arrow[r]\arrow[d] &
        {I}_{Z/\Sigma} \isom \Oh_\Sigma(-d)\arrow[r] & 0 \\
        &                                             & {I}_{Z/X} \arrow[d] & & \\
        &                                             & 0  & & 
    \end{tikzcd}
\end{equation*}
Here the vertical sequence is the sequence of ideal sheaves for the inclusions $Z \subset X \subset \bbP^{n+1}$, while the horizontal
sequence is for the inclusions $Z\subset \Sigma \subset \bbP^{n+1}$. We also have from the local description in the proof of $(1)$ above
that the composite map $${I}_{\Sigma/\bbP^{n+1}} \longrightarrow {I}_{Z/\bbP^{n+1}} \longrightarrow {I}_{Z/X}$$ is surjective.


We will now show that the map ${I}_{X/\bbP^{n+1}} \longrightarrow{I}_{Z/\Sigma}$ is surjective and that it is indeed 
the natural restriction map $\Oh_{\bbP^{n+1}}(-d) \longrightarrow \Oh_\Sigma(-d)$. For this, we complete the diagram above:
\begin{equation*}
    \begin{tikzcd}
    &                  0 \arrow[d]         & 0 \arrow[d] & & \\
                    &  {I}_{\Sigma/\bbP^{n+1}}(-d)\arrow[d] \arrow[r] & {I}_{X/\bbP^{n+1}}  \arrow[d] & & \\
        0 \arrow[r] & {I}_{\Sigma/\bbP^{n+1}}\arrow[d] \arrow[r] & {I}_{Z/\bbP^{n+1}} \arrow[r]\arrow[d] &
        {I}_{Z/\Sigma} \arrow[r] & 0 \\
        &  I_{Z/X} \arrow[r, equal]\arrow[d]                                           & {I}_{Z/X} \arrow[d] & & \\
        &       0                                      & 0  & & 
    \end{tikzcd}
\end{equation*}
It follows from the snake lemma applied to the two vertical short exact sequences that the map 
${I}_{X/\bbP^{n+1}} \longrightarrow {I}_{Z/\Sigma}$ is surjective.

To complete the proof of the claim, we consider the diagram obtained  by various quotient maps $\mathcal{I} \longrightarrow \mathcal{I}/\mathcal{I}^2$
from an ideal of a subvariety to its conormal sheaf: 
\begin{equation*}
    \begin{tikzcd}
        {I}_{X/\bbP^{n+1}} \arrow[hook]{r}\arrow[two heads]{d} & {I}_{Z/\bbP^{n+1}}\arrow[two heads]{r}
        \arrow[two heads]{d} & {I}_{Z/\Sigma}\arrow[two heads]{d} \\
        N^{\vee}_{X/\bbP^{n+1}} \arrow[r] & N^\vee_{Z/\bbP^{n+1}} \arrow[r] & N^\vee_{Z/\Sigma} 
    \end{tikzcd}
\end{equation*}
The composition of the top horizontal arrows is a surjection as noted in the above paragraph. 
Since the right most vertical arrow
is a surjection, this means that the composition of the bottom horizontal arrows is also a surjection. Consequently, the composition
$$N^{\vee}_{X/\bbP^{n+1}}\tensor\Oh_Z \longrightarrow N^\vee_{Z/\bbP^{n+1}} \longrightarrow  N^\vee_{Z/\Sigma} $$
is also a surjection, and hence an isomorphism since they are both isomorphic to $\Oh_Z(-d)$.

(3) This is \cite[Lemma 2.2]{MPRV}.
\end{proof}
The key to our analysis is the following (although we will not use (4) in the sequel):
\begin{cor}\label{mprvcor}
Assume in the situation of Lemma \ref{l3} that $Z\subset X$ (equivalently $E$ when the conditions of Proposition \ref{serre} are satisfied) is ACM. 
Then the following are equivalent:
\begin{enumerate}
\item $E$ is a direct sum of line bundles.
\item $Z$ is extendable.
\item The normal bundle sequence \eqref{nbsa} for the inclusions $Z \subset X \subset \bbP^{n+1}$  splits.
\item There exists a subscheme $Z_{(1)}\supset Z$ in the first-order thickening $X_{(1)}$ of the hypersurface $X$ in
$\bbP^{n+1}$ such that the following sequence is exact 
$$I_{Z_{(1)}/X_{(1)}}(-d)\stackrel{f}\longrightarrow I_{Z_{(1)}/X_{(1)}}\longrightarrow I_{Z/X}\longrightarrow 0.$$
Furthermore, $fI_{Z_{(1)}/X_{(1)}}(-d)= I_{Z/X}(-d)$.
\end{enumerate}
\end{cor}

\begin{proof}
Indeed, by Lemma \ref{l3}, all that remains to be proved is that $(4)$ implies $(1)$ under the ACM hypothesis. We
recall from \cite[Section 2]{MPRV} that there exists a short exact sequence 
\begin{equation}\label{prevex1}
    0\longrightarrow G\longrightarrow F\longrightarrow I_{Z/X}\longrightarrow 0
\end{equation}
such that
\begin{itemize}
    \item  $F$ is a direct sum of line bundles, 
    \item $\HH^0_*(X,F)\longrightarrow \HH^0_*(X,I_{Z/X})$ is surjective, and 
    \item $G$ is ACM.
\end{itemize}  
By \cite[Lemma 2.3]{MPRV}, $G$ extends to a bundle $\mathcal{G}$ on $X_{(1)}$. Applying \cite[Proposition 2.4]{MPRV}, we see that 
$G$ is a direct sum of line bundles. 
Since $F$ is a direct sum of line bundles, we conclude that the map $$\HH^0(X,F^{\vee}\otimes E(-e))\longrightarrow \HH^0(X,F^{\vee}\otimes I_{Z/X})$$ induced by the exact sequence 
\begin{equation}\label{prev'}
    0 \longrightarrow \Oh_X(-e)^{\oplus r-1} \longrightarrow E(-e) \longrightarrow I_{Z/X} \longrightarrow 0 
\end{equation}
is surjective as $\HH^1(X,F^{\vee}\otimes\Oh_X(-e))=0$. Thus the map $F\longrightarrow I_{Z/X}$ in \eqref{prevex1} lifts to a map $F\longrightarrow E(-e)$. Consequently, defining $$\widetilde{F}:=F\oplus \Oh_X(-e)^{\oplus r-1}$$ and using snake lemma, we obtain the following diagram with exact rows and columns:
\begin{equation}\label{big1}
\begin{tikzcd}
    & & 0\arrow[d] & 0\arrow[d] &\\
    & & \Oh_X(-e)^{\oplus r-1}\arrow[d]\arrow[r,equal] & \Oh_X(-e)^{\oplus r-1}\arrow[d]\\
    0\arrow[r] & G\arrow[r]\arrow[d, equal] & \widetilde{F}\arrow[d]\arrow[r] & E(-e)\arrow[d]\arrow[r] & 0\\
    0\arrow[r] & G\arrow[r] & F\arrow[r]\arrow[d] & I_{Z/X}\arrow[r]\arrow[d] & 0\\
    & & 0 & 0 &
\end{tikzcd}
\end{equation}
Since $G$ is a direct sum of line bundles, we obtain $\textrm{Ext}^1(E(-e),G)=0$ (as $E$ is ACM). Consequently the middle row of \eqref{big1} is split. Since $\widetilde{F}$ is a direct sum of line bundles, so is $E$.
\end{proof}

\section{Surjectivity via Green's theorem}\label{green}
{\color{black}
We now proceed to prove the main technical result that is needed in the proof of Theorem \ref{mainthm1}. 
Throughout this section unless stated otherwise, $X\subset \mathbb{P}^5$ is a general hypersurface of degree $d$, and $Z\subset X$ is an 
ACM local complete intersection surface. 
We also assume that
\begin{itemize}
    \item[(A)] {$I_{Z/X}(e+1)$ is globally generated, and}
    \item[(B)] there is a smooth member
$\Theta \in \vert I_{Z/X}(e)\vert$ (in particular $e\geq 1$).
\end{itemize} 

By assumption $(B)$, we have inclusions
$$Z \subset \Theta \subset X \subset \bbP^5$$ 
and the corresponding normal bundle sequence
\begin{equation}\label{nbs1}
0 \longrightarrow N_{Z/\Theta} \longrightarrow N_{Z/X} \longrightarrow 
N_{{\Theta/X}|Z} \longrightarrow 0.
\end{equation}
Since $N_{\Theta/X} \isom \Oh_\Theta(e)$, taking determinants, we have the identification 
$$N_{Z/\Theta}  \isom  \det{N_{Z/X}}\tensor\Oh_Z(-e)\isom \omega_Z\tensor{\omega_\Theta}^{-1} =  \ell$$
whence the normal bundle sequence in \eqref{nbs1} may be rewritten as
\begin{equation}\label{nbs2}
0 \longrightarrow \ell \longrightarrow N_{Z/X} \longrightarrow \Oh_Z(e) \longrightarrow 0.
\end{equation}             
Taking cohomology, we get the sequence
$$0 \longrightarrow\HH^0(Z, \ell) \longrightarrow \HH^0(Z, N_{Z/X}) {\overset{\alpha}{\longrightarrow}} \HH^0(Z, \Oh_Z(e)) \longrightarrow \cdots .$$

Setting $W:= \textrm{Image}(\alpha)$, we have an exact sequence
\begin{equation*}
0 \longrightarrow \HH^0(Z, \ell) \longrightarrow \HH^0(Z, N_{Z/X}) \longrightarrow W \longrightarrow 0.
\end{equation*}
More generally, twisting \eqref{nbs2} with $\Oh_Z(b)$ for any  $b\in\bbZ$, we also have exact sequences
\begin{equation}\label{coh_nbsz}
0 \longrightarrow \HH^0(Z, \ell(b)) \longrightarrow \HH^0(Z, N_{Z/X}(b)) \longrightarrow W_{b+e} \longrightarrow 0,
\end{equation}
where $$W_{b+e}:= \textrm{Image}\left[~\HH^0(Z, N_{Z/X}(b)) \longrightarrow\HH^0(Z, \Oh_Z(b+e))~\right].$$
Evidently $W=W_e$ in the above notation.
}
{\color{black}
\begin{lemma}\label{nisgg}
The vector spaces $W_{b+e}$ for $b > 0$ are base point free linear subsystems of the space of global sections
$\HH^0(Z, \Oh_Z(b+e))$.                                                                    
\end{lemma}

\begin{proof}
We have commutative diagrams
\begin{equation*}
\begin{tikzcd}
\HH^0(Z, N_{Z/X}(b))\tensor\Oh_Z \arrow[r, two heads]\arrow[d] & W_{b+e}\tensor\Oh_Z\arrow[d] \\
N_{Z/X}(b)\arrow[r, two heads] & \Oh_Z(b+e)
\end{tikzcd}
\end{equation*}
with surjective horizontal maps. {Arguing as in the proof of 
\cite[Proposition 2.1]{BR} (which expands on results implicit in \cite{Pa, 
VoisinJDG})}, we see that $N_{Z/X}(b)$ is globally generated for $b> 0$ and 
hence the left vertical arrow is surjective for $b> 0$. This implies that 
the right vertical map is also surjective, i.e., $W_{b+e}$ is a base point 
free linear subsystem of $\HH^0(Z, \Oh_Z(b+e))$ for $b>0$.
\end{proof}                                                                     
}
{\color{black}
We have the following preliminary result which will be used in the proof of 
Theorem \ref{mainthm2}.

\begin{prop}\label{genericallybpf}
Let $X \subset \bbP^5$ be a smooth hypersurface of degree $d$ and $Z \subset 
X$ be a local complete intersection surface. Under the assumptions {$(A)$ 
and $(B)$} above, there is a linear subsystem $\widetilde{W} \subset 
\HH^0(\bbP^5, \Oh_{\bbP^5}(e+1))$ whose base locus is supported on a finite 
set in the complement $\bbP^5\setminus{X}$.
\end{prop}

\begin{proof}
We have the sequence 
$$0 \longrightarrow I_{Z/X}(e+1) \longrightarrow \Oh_X(e+1) \longrightarrow \Oh_Z(e+1) \longrightarrow 0.$$
Taking global sections, we get an exact sequence
$$ 0 \longrightarrow \HH^0(X, I_{Z/X}(e+1)) \longrightarrow \HH^0(X, \Oh_X(e+1)) \longrightarrow 
\HH^0(Z, \Oh_Z(e+1) \longrightarrow 0.$$
Note that right exactness follows from the fact that $Z$ is ACM.
Let $\widetilde{W}_X$ denote the lift of $W_{e+1}$ under the surjective map
\[
\begin{tikzcd}
    \HH^0(X, \Oh_{X}(e+1)) \arrow[r, two heads]& \
    \HH^0(Z, \Oh_{Z}(e+1)),
\end{tikzcd}
\]
so that we have an exact sequence
$$ 0\longrightarrow \HH^0(X, I_{Z/X}(e+1)) \longrightarrow \widetilde{W}_X \longrightarrow W_{e+1} \longrightarrow 0.$$
The linear system $\widetilde{W}_X$ is base point free. Indeed, it has no base point on $Z$ by Lemma  \ref{nisgg}. 
{Since $\widetilde{W}_X$ contains the linear subsystem 
$\HH^0(X, I_{Z/X}(e+1))$, using (A) we conclude that $\widetilde{W}_X$
has no base points in the complement $X\setminus Z$  either}.

Now let $\widetilde{W}$ denote the lift of $W_{e+1}$ under the surjective 
map
\[
\begin{tikzcd}
    \HH^0(\bbP^5, \Oh_{\bbP^5}(e+1)) \arrow[r, two heads]& \
    \HH^0(Z, \Oh_{Z}(e+1)),
\end{tikzcd}
\]
so that we have a diagram
\[
\begin{tikzcd}
0 \arrow[r] & \HH^0(\bbP^5, {I}_{Z/\bbP^5}(e+1)) \arrow[r]\arrow[d, 
two heads] & \widetilde{W} \arrow[d]\arrow[r, two heads]& W_{e+1} \arrow[d, 
equals]\arrow[r] & 0 \\
0 \arrow[r] & \HH^0(X, {I}_{Z/X}(e+1)) \arrow[r] & 
\widetilde{W}_X \arrow[r, two heads] & W_{e+1} \arrow[r] & 0.
\end{tikzcd}
\]
In the above, the surjection of the leftmost vertical map follows from the exact sequence 
$$0\longrightarrow\Oh_{\mathbb{P}^5}(e+1-d)\longrightarrow I_{Z/\mathbb{P}^5}(e+1)\longrightarrow I_{Z/X}(e+1)\longrightarrow 0.$$
By the snake lemma, it follows that the middle arrow
$$\widetilde{W} \longrightarrow \widetilde{W}_X$$
is also a surjection (in fact, for $e+1 < d$, this map is an isomorphism).

Let $T$ denote the base locus of the linear system $\widetilde{W}$. From the 
surjection above, it follows that $T$ does not meet the hypersurface $X$. 
Since $X$ is ample, we conclude that $T$ is a finite set in the complement 
$\bbP^5\setminus{X}$. This finishes the proof.
\end{proof}


With hypotheses as in Proposition \ref{genericallybpf}, let 
$C \subset Y \subset \bbP^4$
be general hyperlane sections of the inclusions $Z \subset X \subset \bbP^5$ 
above. Define the subspaces $W_{b+e, C} \subset \HH^0(C, \Oh_C(b+e))$ as the image under the composite map 
\[
W_{b+e} \into \HH^0(Z, \Oh_Z(b+e)) \onto \HH^0(C, \Oh_C(b+e)).
\]

\begin{claim}\label{c0}
    $W_{e+1, C}$ is base point free.
\end{claim}
\noindent\textit{Proof of Claim \ref{c0}:} For $b=1$, we have a commutative square
\begin{equation}\label{bpfC}
\begin{tikzcd}
W_{e+1}\tensor\Oh_Z \arrow[r, two heads]\arrow[d, two heads] &
W_{e+1, C}\tensor\Oh_C \arrow[d] \\
\Oh_Z(e+1) \arrow[r, two heads] & \Oh_C(e+1)
\end{tikzcd}
\end{equation}
where the left vertical map is surjective by Lemma \ref{nisgg}, and the bottom horizontal map, being the 
restriction map, is surjective. 
Consequently it follows that the right vertical arrow is a surjection as well.\hfill$\spadesuit$

\smallskip

As before, let $\widetilde{W}_{\bbP^4}$ be the lift of $W_{e+1, C}$ under the map
\[
\begin{tikzcd}
    \HH^0(\bbP^4, \Oh_{\bbP^4}(e+1)) \arrow[r, two heads]& \HH^0(C, \Oh_{C}(e+1)).
\end{tikzcd}
\]
We then have a commutative diagram with exact rows
\[
\begin{tikzcd}
0 \arrow[r] & {I}_{Z/\bbP^5}(e+1) \arrow[r]\arrow[d, two heads] & \Oh_{\bbP^5}(e+1) \arrow[d, two heads]\arrow[r]& \Oh_{Z}(e+1) \arrow[d, 
two heads]\arrow[r] & 0 \\
0 \arrow[r] & {I}_{C/\bbP^4}(e+1) \arrow[r] & 
\Oh_{\bbP^4}(e+1) \arrow[r, two heads] & \Oh_{C}(e+1) \arrow[r] & 0.
\end{tikzcd}
\]
which gives the exact sequence
$$0\longrightarrow
I_{Z/\mathbb{P}^5}(e)\longrightarrow I_{Z/\mathbb{P}^5}(e+1)\longrightarrow I_{C/\bbP^4}(e+1)\longrightarrow 0.$$
Consequently, we obtain the following surjection as $Z\subset\mathbb{P}^5$ is ACM:
\[
\begin{tikzcd}
    \HH^0(\bbP^5, {I}_{Z/\bbP^5}(e+1))\arrow[r, two heads]&\HH^0(\bbP^4, {I}_{C/\bbP^4}(e+1)).
\end{tikzcd}
\]


Now consider the commutative diagram
\[
\begin{tikzcd}
0 \arrow[r] & \HH^0(\bbP^5, {I}_{Z/\bbP^5}(e+1)) \arrow[r]\arrow[d, 
two heads] & \widetilde{W} \arrow[d]\arrow[r, two heads]& W_{e+1} \arrow[d, 
two heads]\arrow[r] & 0 \\
0 \arrow[r] & \HH^0(\bbP^4, {I}_{C/\bbP^4}(e+1)) \arrow[r] & 
\widetilde{W}_{\bbP^4} \arrow[r, two heads] & W_{e+1, C} \arrow[r] & 0.
\end{tikzcd}
\]
The middle vertical arrow is surjective since the left and right verticals 
are surjective.

\begin{lemma}
With notation as above, $\widetilde{W}_{\bbP^4}\subset \HH^0(\bbP^4, \Oh_{\bbP^4}(e+1))$
is a base point free linear subsystem. 
\end{lemma}

\begin{proof}
Recall from Proposition \ref{genericallybpf} that 
$\widetilde{W}$ is base point free away from a finite set $T \subset 
\bbP^5\setminus{X}$. Since a general hyperplane $\bbP^4 \subset \bbP^5$ will 
avoid the finite set $T$, and $\widetilde{W}\longrightarrow \widetilde{W}_{\bbP^4}$ is a surjection by the previous discussion, it follows that 
$\widetilde{W}_{\bbP^4}$ is base point free.
\end{proof}

\begin{prop}\label{gg}
With notation as above, if $\binom{e+5}{4} \leq 2d-4$, then the 
multiplication map
$$W_{d+e-5, C}\tensor\HH^0(C, \Oh_C(d)) \longrightarrow 
\HH^0(C, \Oh_C(2d+e-5))$$
is surjective.
\end{prop}
}
Before we work out the proof of Proposition \ref{gg}, we recall the 
following result of Green which will play a key role for us:

 \begin{thm}[{\cite[Theorem 2]{Green}}]\label{green_koszul}
Let $\widetilde{W} \subset \HH^0(\bbP^r, \Oh_{\bbP^r}(a))$ be a base point free linear system. Then the Koszul complex
\[
\bigwedge\limits^{p+1}\widetilde{W}\tensor \HH^0(\bbP^r, \Oh_{\bbP^r}(k-a)) \longrightarrow \bigwedge\limits^{p}\widetilde{W}\tensor \HH^0(\bbP^r, \Oh_{\bbP^r}(k)) \longrightarrow
\bigwedge\limits^{p-1}\widetilde{W}\tensor \HH^0(\bbP^r, \Oh_{\bbP^r}(k+a))
\]
is exact in the middle provided that 
$ \mathrm{codim}(\widetilde{W}) \leq k - p - a.$
 \end{thm}
{\color{black}
\begin{proof}[Proof of Proposition \ref{gg}]

We have a commutative square where the vertical maps are the multiplication maps, and the horizontal maps  
are the restriction maps:
\begin{equation*}
    \begin{tikzcd}
    \widetilde{W}_{\bbP^4}\tensor\HH^0(\bbP^4, \Oh_{\bbP^4}(2d-6))\arrow[r]\arrow[d, swap,"\widetilde{\mu}"] & W_{e+1, C}\tensor\HH^0(C, \Oh_{C}(2d-6))\arrow[d,"\mu"] \\
\HH^0(\bbP^4, \Oh_{\bbP^4}(2d+e-5))\arrow[r, two heads]& \HH^0(C, \Oh_{C}(2d+e-5))
\end{tikzcd}
\end{equation*}
\begin{claim}\label{c1}
Under the assumptions of Proposition \ref{gg}, the map 
$$\mu:W_{e+1, C}\tensor\HH^0(C, \Oh_{C}(2d-6))\longrightarrow\HH^0(C, \Oh_{C}
(2d+e-5))$$ is surjective.    
\end{claim}
\noindent{\it Proof of Claim \ref{c1}:} First, recall that 
$\widetilde{W}_{\bbP^4}$ is base point free. In Green's result (Theorem 
\ref{green_koszul} above),  letting $p = 0$, $k=2d+e-5$, and $a=e+1$,
we see that if 
\begin{equation}\label{need}
    \mbox{codim}(W_{e+1, C})\ = \ \mbox{codim}(\widetilde{W}_{\bbP^4})\\  \leq \ 2d-6,
\end{equation}
then the left vertical map $\widetilde{\mu}$ is surjective which implies 
that the right vertical map $\mu$ is surjective as well. To prove 
\eqref{need}, note that $W_{e+1, C}$ is a base point free subspace of 
$\HH^0(C,\Oh_C(e+1))$ by Claim \ref{c0}. 
As the restriction map 
$$\HH^0(\bbP^4, \Oh_{\bbP^4}(e+1)) \longrightarrow \HH^0(C, \Oh_C(e+1))$$
is a surjection,  it follows that 
$$\mbox{codim}(W_{e+1, C}) \leq h^0(\Oh_{\bbP^4}(e+1)) - 2 
= \binom{e+5}{4}-2.$$
Thus, \eqref{need} holds since $\binom{e+5}{4} \leq 2d -4$, whence the multiplication map $\mu$ is surjective.\hfill$\spadesuit$

\smallskip

We continue with the proof of Proposition \ref{gg}. Recall that $d>6$ by hypothesis as $e\geq 1$.
We have the commutative diagram
\begin{equation*}
\begin{tikzcd}
{W}_{e+1, C}\tensor\HH^0(C, \Oh_{C}(1))^{\tensor (2d-6)}\arrow[r, two heads]\arrow[d, swap, "\mu_1"] & 
W_{e+1, C}\tensor\HH^0(C, \Oh_{C}(2d-6))\arrow[d, "\mu"] \\
W_{d+e-5, C}\tensor\HH^0(C, \Oh_{C}(d))\arrow[r, "\mu_d"] & \HH^0(C, \Oh_{C}(2d+e-5)). 
\end{tikzcd}
\end{equation*}
That the top horizontal map is surjective follows by a diagram similar to \eqref{multiplicationmap}.
The surjectivity of $\mu_d$ now follows by the surjectivity of $\mu$ proven in Claim \ref{c1}.
\end{proof}

}

\section{Proof of Theorem \ref{mainthm2} via the Beauville-M\'erindol criterion}
We recall a very elegant splitting criterion, due to Beauville and M\'erindol (see \cite[Lemme 1]{BM87})
for a sequence of vector bundles on a curve to be split.
Since the proof is very short, we include it to enhance the ease of reading.

\begin{lemma}[The Beauville-M\'erindol criterion]
Let $C$ be a local complete intersection projective curve and 
\begin{equation}\label{sess}
    0 \longrightarrow E \longrightarrow F \longrightarrow G \longrightarrow 0 
\end{equation}
be a short exact sequence of bundles. This sequence splits if 
\begin{enumerate}
\item[(i)]  $\HH^0(C, F) \longrightarrow \HH^0(C, G)$ is surjective, and
\item[(ii)] the cup product map
$$ \cup : \HH^0(C, G)\tensor\HH^0(C, E^\vee\tensor\omega_C) \longrightarrow \HH^0(C, E^\vee\tensor{G}\tensor\omega_C)$$
is surjective.
\end{enumerate}
\end{lemma}

\begin{proof}
We first note that the boundary map  $\HH^0(C, G) {\overset{\partial}{\longrightarrow}}  \HH^1(C, E)$ 
yields the map
$$\partial : \HH^0(C, G)\tensor\HH^0(C, E^\vee\tensor\omega_C)  \longrightarrow \bbC.$$
The short exact sequence \eqref{sess} corresponds to an element
$\eta \in \Ext^1(G, E) \isom \HH^1(C, G^\vee\tensor{E}),$ and via Serre duality, we 
treat the element $\eta$ as a map
$$\eta: \HH^0(C, G\tensor{E}^\vee\tensor\omega_C) \longrightarrow \bbC.$$
To this end, we note the following commutative diagram
\begin{equation}
\begin{tikzcd}
\HH^0(C, G)\tensor\HH^0(C, E^\vee\tensor\omega_C)\arrow[r, "\partial"]\arrow[d, swap, "\cup"] & \bbC \\
\HH^0(C, G\tensor{E}^\vee\tensor\omega_C)\arrow[ur,swap,"\eta"]
\end{tikzcd}
\end{equation}
Consequently, we have that $\partial = \eta\circ\cup$. Since the cup product map $\cup$ is surjective, 
we have
$\eta = 0 \iff \partial = 0$, and the latter is zero by assumption.
\end{proof}


\comment{
\begin{lemma}\label{morep}
Let $X\subset \mathbb{P}^4$ be a smooth hypersurface of degree $d$, and $C\subset X$ be an ACM local complete intersection curve. 
Assume that {$I_{C/X}(e+1)$ is globally generated} but $I_{C/\mathbb{P}^4}(e+1)$ is not globally generated. Then $C\subset X$ is extendable.
\end{lemma}

\begin{proof}
By assumption, we have a surjection 
$$\Oh_X^{\oplus M} \onto I_{C/X}(e+1)$$
that leads to the following diagram
\[
\begin{tikzcd}
0\arrow[r] & \Oh_{\bbP^4}(-d)^{\oplus M}\arrow[r, "f"]\arrow[d] & \Oh_{\bbP^4}^{\oplus M}\arrow[r]\arrow[d] & 
\Oh_X^{\oplus M}\arrow[r]\arrow[d] & 0 \\
0\arrow[r]  & \Oh_{\bbP^{4}}(e+1-d)\arrow[r, "f"] & I_{C/\bbP^{4}}(e+1)\arrow[r] &  I_{C/X}(e+1)\arrow[r] & 0.
\end{tikzcd}
\]
Since the middle vertical arrow is not a surjection, its image is $\mathcal{J}(e+1)$ where $\mathcal{J}$ 
is an ideal sheaf defining a subscheme $\Sigma$ in $\bbP^4$. Note that $\mathcal{J}=\widetilde{P}$, the 
sheafification of the homogeneous ideal 
$$P:=\textrm{Image}(\HH^0_*(\bbP^4, \Oh_{\bbP^4}(-e-1)^{\oplus M})\longrightarrow
\HH^0_*(\bbP^4,I_{C/\bbP^{4}}))$$ 
of $S:=\mathbb{C}[X_0,\cdots, X_4]$. 

Let $X=\left\{f=0\right\}$, and we sometimes write $(f)$ for $I_{X/\mathbb{P}^4}$. 
One can argue as in the proof of Lemma \ref{l3}(1) that $\Sigma\cap X=C$, or equivalently 
$\mathcal{J}+(f)=I_{C/\mathbb{P}^4}$. Note that $f\notin P$. Indeed, if $f\in P$, sheafifying we get $(f)\subseteq \mathcal{J}$, whence $\mathcal{J}=I_{C/\mathbb{P}^4}$, a contradiction.


To prove the lemma, we will show that $C$ {\it extends infinitesimally} in the sense of  
Corollary \ref{mprvcor}(4). To this end, we set $\mathcal{I}_1:= \mathcal{J} + (f^2)$.

\begin{claim}\label{ne}
    $\mathcal{J} \subset \mathcal{I}_1 \subsetneq {I}_{C/\mathbb{P}^4}.$
\end{claim}
\noindent{\it Proof of Claim 2:}
We only need to show that $\mathcal{I}_1\subsetneq I_{C/\mathbb{P}^4}$. This follows from the fact that $$P+(f^2)\subsetneq \HH^0_*(\bbP^4,I_{C/\mathbb{P}^4}).$$ Indeed, $f\in \HH^0_*(I_{C/\mathbb{P}^4})$ but $f\notin P+(f^2)$ as $P$ is homogeneous and $f\notin P$. The assertion follows from sheafifying the above.\hfill$\spadesuit$

\smallskip

Let $C_{(1)}\subset\mathbb{P}^4$ be the subscheme defined by the ideal $\mathcal{I}_1$, and let $X_{(1)}\subseteq\mathbb{P}^4$ be the subscheme defined by $f^2=0$. Then it is straightforward to verify that 
$C\subset C_{(1)}\subset X_{(1)}$ and that $C_{(1)}\cap X=C$. 

Consider the following diagram with exact rows:
\[
\begin{tikzcd}
0\arrow[r] & I_{C_{(1)}/X_{(1)}}(-d)\arrow[r]\arrow[d, "f"] & 
\Oh_{X_{(1)}}(-d)\arrow[r]\arrow[d, "f"] & \Oh_{C_{(1)}}(-d)
\arrow[r]\arrow[d, "f"] & 0 \\
0\arrow[r] & I_{C_{(1)}/X_{(1)}}\arrow[r]\arrow[d] & 
\Oh_{X_{(1)}}\arrow[r]\arrow[d] & \Oh_{C_{(1)}}
\arrow[r]\arrow[d] & 0 \\
0\arrow[r] & I_{C/X}\arrow[r]\arrow[d] & 
\Oh_{X}\arrow[r]\arrow[d] & \Oh_{C}
\arrow[r]\arrow[d] & 0 \\
 & 0 & 0 &  0 & .
\end{tikzcd}
\]

\begin{claim} In the above diagram, we have:
\begin{itemize}
    \item[(i)] The middle vertical sequence is exact with $${\rm Image}~[f:\Oh_{X_{(1)}}(-d) \longrightarrow \Oh_{X_{(1)}}] = \Oh_{X}(-d)$$ via the injection $\mathcal{O}_X(-d)\xrightarrow[]{f}\Oh_{X_{(1)}}$.
    \item[(ii)] The right vertical sequence is exact with $${\rm Image}~[f:\Oh_{C_{(1)}}(-d) \longrightarrow \Oh_{C_{(1)}}] = \Oh_{C}(-d)$$ via the injection $\mathcal{O}_C(-d)\xrightarrow[]{f}\Oh_{C_{(1)}}$.
    \item[(iii)] The left vertical sequence is exact with $${\rm Image}~[f: I_{C_{(1)}/X_{(1)}}(-d) \longrightarrow I_{C_{(1)}/X_{(1)}}] = I_{C/X}(-d)$$ via the injection $I_{C/X}(-d)\xrightarrow[]{f}I_{C_{(1)}/X_{(1)}}$.
\end{itemize}
\end{claim}

\noindent{\it Proof of Claim 3:}
We denote the local equation of $X\subset\mathbb{P}^4$ by the same symbol $f=0$ in Spec$(A)$. Let the ideals of $C, C_{(1)}, \Sigma$ in Spec$(A)$ be $I, I_1, J$ respectively. Recall that $$J+(f)=I\,\textrm{ and }\, J+(f^2)=I_1,\, \textrm{ in particular }\, I=I_1+(f).$$

We have (i) as $A/(f)\xrightarrow[]{f} A/(f^2)$ is injective and $$f\cdot(A/(f^2)) = (f)/(f^2) \isom f(A/(f)).$$ 

To see (ii), consider the map 
\begin{equation}\label{tm}
    {f}: A \longrightarrow A/I_1, \,\, a \mapsto [af]:= af + I_1.
\end{equation} 
It is evident that the the image of this map is contained in $I/I_1$. 
Let $i \in I$; then $i=i_1 + af$ for some $i_1 \in I_1$ and $a\in A$. This shows that 
$a\in A$ is a preimage of $[i]\in I/I_1$ under the above map, which is therefore surjective onto 
$I/I_1$. Moreover
$fi = fi_1 + af^2 \in I_1$ whence $I \subset {\rm Ker}({f})$. 
Thus we have a surjection
$${f}: A/I \onto I/I_1.$$
Globally, this gives the first surjection below:
$$\Oh_C(-d) \onto I_{C/\mathbb{P}^4}/\mathcal{I}_1 \into \Oh_{C_{(1)}}.$$
\textcolor{red}{The surjectivity of the first map implies that the support of $I_{C/\mathbb{P}^4}/\mathcal{I}_1$ is a subscheme of $C$ which is non-empty by Claim \ref{ne}. Observe that the first map has 
non-trivial kernel if and only if this subscheme is a proper non-empty subscheme of $C$. Since
$C$ is irreducible and reduced, were the support to be a proper non-empty subscheme of $C$, then 
$I_{C/\mathbb{P}^4}/\mathcal{I}_1$ would be a torsion sheaf on $C_{(1)}$ which would contradict the injectivity of the second map.
Hence the first map must be an isomorphism. Thus we have the short exact sequence
$$0 \longrightarrow \Oh_C(-d) \stackrel{f}{\longrightarrow} \Oh_{C_{(1)}} \longrightarrow \Oh_C \longrightarrow 0.$$}
Finally, note that \eqref{tm} also induces a surjection $A/I_1\onto I/I_1$.

To prove (iii), note that the fact that ${\rm Ker}(I_{C_{(1)}/X_{(1)}}\longrightarrow I_{C/X})$ can be 
identified with $I_{C/X}(-d)$ via the injection $I_{C/X}(-d)\xrightarrow[]{f}I_{C_{(1)}/X_{(1)}}$ follows 
from (i), (ii) and the snake lemma. The conclusion follows by noting that
$$f(I_1/(f^2)) = f(J+(f^2)/(f^2)) = fJ + (f^2)/(f^2) \isom f(I/(f)),$$ which finishes the proof of the 
claim.\hfill$\spadesuit$

\smallskip

By Corollary \ref{mprvcor} and part (iii) of the above claim, it follows that $C\subset X$ is extendable. 
\end{proof}
}

\begin{proof}[Proof of Theorem \ref{mainthm2}]
{\color{black} By taking a general hyperplane section of the inclusions 
$$Z \subset \Theta \subset X \subset \bbP^5$$ 
we obtain inclusions 
$$C \subset S \subset Y \subset \bbP^4.$$
Here $S$ is a smooth complete intersection surface.
}

Recall that the normal bundle $N_{C/Y}$ is a rank $2$ bundle and as such we 
have
$$N_{C/Y}^\vee \isom N_{C/Y}\tensor\left(\det{N_{C/Y}}\right)^{-1} \isom 
N_{C/Y}\tensor{\omega_Y}\tensor\omega_{C}^{-1}.$$
Consequently, 
\begin{equation}\label{ndual}
N_{C/Y}^\vee\tensor\omega_C  \isom N_{C/Y}\tensor{\omega_Y}.
\end{equation}
By the Beauville-M\'erindol criterion, we need to check that
\begin{enumerate}
\item[(a)] the map $\alpha: \HH^0(C, N_{C/\bbP^4}) \longrightarrow \HH^0(C, \Oh_C(d))$ 
is surjective, and 
\item[(b)] the cup product map
$$\HH^0(C, \Oh_C(d))\tensor\HH^0(C, N_{C/Y}^\vee\tensor\omega_C) 
\longrightarrow\HH^0(C, N_{C/Y}^\vee\tensor\omega_C(d))$$
is surjective.
\end{enumerate}
Since $Y$ is a {\it general} hypersurface of degree $d$ in $\bbP^4$,  
we have (see, for example, \cite[Proposition 3.2]{BK})
\small
$$\textrm{Image}\left[\HH^0(\bbP^4, \Oh_{\bbP^4}(d)) \longrightarrow 
\HH^0(C, \Oh_C(d))\right] \subset
\textrm{Image}\left[\HH^0(\bbP^4, N_{C/\bbP^4}(d)) \longrightarrow \HH^0(C, 
\Oh_C(d))\right].$$
\normalsize
Recall that $C$ is ACM, whence the map 
$$\HH^0(\bbP^4, \Oh_{\bbP^4}(d)) \longrightarrow\HH^0(C, \Oh_C(d))$$
is surjective, which verifies condition (a).

For (b), using the identification in \eqref{ndual}, we are reduced to proving that
the cup product map
$$\HH^0(C, \Oh_C(d))\tensor\HH^0(C, N_{C/Y}(d-5)) \longrightarrow \HH^0(C, N_{C/Y}(2d-5))$$
is surjective. Let us define
$$V_d:= \HH^0(C, \Oh_C(d)).$$
{\color{black}
The normal bundle sequence for the inclusions 
$C\subset S \subset Y$ 
\begin{equation}\label{nbsc}
0 \longrightarrow \ell_C:= \ell\tensor\Oh_C \longrightarrow N_{C/Y} \longrightarrow \Oh_C(e) \longrightarrow 0
\end{equation}
is obtained by restricting the normal sequence \eqref{nbs2}
\begin{equation*}
0 \longrightarrow \ell \longrightarrow N_{Z/X} \longrightarrow \Oh_Z(e) 
\longrightarrow 0.
\end{equation*}             
Taking cohomology, we get the sequence
$$0 \longrightarrow\HH^0(C, \ell_C) \longrightarrow \HH^0(C, N_{C/Y}) 
{\overset{\bar\alpha}{\longrightarrow}} \HH^0(C, \Oh_C(e)) \longrightarrow 
\cdots .$$

Setting $\overline{W}:= \textrm{Image}(\bar\alpha)$, 
we have an exact sequence
\begin{equation*}
0 \longrightarrow \HH^0(C, \ell_C) \longrightarrow \HH^0(C, N_{C/Y}) 
\longrightarrow \overline{W} \longrightarrow 0.
\end{equation*}
More generally, twisting \eqref{nbsc} with $\Oh_C(b)$ for any  $b\in\bbZ$, 
we also have exact sequences
\begin{equation}\label{coh_nbsc}
0 \longrightarrow \HH^0(C, \ell_C(b)) \longrightarrow 
\HH^0(C, N_{C/Y}(b)) \longrightarrow \overline{W}_{b+e} \longrightarrow 0,
\end{equation}
where $$\overline{W}_{b+e}:= \textrm{Image}\left[~\HH^0(C, N_{C/Y}(b)) 
\longrightarrow\HH^0(C, \Oh_C(b+e))~\right].$$
As before, it is evident that $\overline{W}=\overline{W}_e$ in the above 
notation.

We have the following commutative diagram, where the vertical maps are 
the multiplication maps
\small
\begin{equation}\label{finale}
\begin{tikzcd}
0\arrow[r] &  \HH^0(C, \ell_C(d-5))\tensor{V_d}\arrow[r]\arrow[d] & 
\HH^0(C, N_{C/Y}(d- 5))\tensor{V_d}\arrow[r]\arrow[d] & 
\overline{W}_{d+e-5}\tensor{V_d}\arrow[r]\arrow[d] & 0 \\
0\arrow[r]  & \HH^0(C, \ell_C(2d-5))\arrow[r] & 
\HH^0(C, N_{C/Y}(2d-5))\arrow[r] &  
\overline{W}_{2d+e-5}\arrow[r] & 0.
\end{tikzcd}
\end{equation}
\normalsize
To prove $(b)$ above, which is to say that the vertical map in the middle in 
the above diagram is surjective, it suffices to prove that the left and the 
right vertical arrows are surjective which we now proceed to prove.

\smallskip

\noindent{\it The left vertical map:} We first note that
$d> 5$ by assumption. Next, we have the exact 
sequence defining the ACM bundle $E$ associated to $Z$
$$0 \longrightarrow \Oh_X^{\oplus r-1} \longrightarrow E \longrightarrow I_{Z/X}(e) \longrightarrow 0.$$
Restricting this sequence to the hyperplane section $Y\subset X$, we get
the sequence
$$0 \longrightarrow \Oh_Y^{\oplus r-1} \longrightarrow E\tensor\Oh_Y \longrightarrow I_{C/Y}(e) \longrightarrow 0.$$
Note that $E\tensor\Oh_Y$ is also ACM. 
Taking duals and arguing as in the proof of Proposition \ref{serre}, we 
see that $\ell_C$ is also generated in degree $0$. The surjectivity of the 
left vertical map now follows from Corollary \ref{2gen}.

\smallskip

\noindent{\it The right vertical map:} Note that by definition
$W_{d+e-5, C} \subset \overline{W}_{d+e-5}$, whence 
the surjectivity of the right vertical map
now follows from Proposition \ref{gg}.
}
\end{proof}
We end this section with the following observation in the $n=3$ case.
{\color{black} 
\begin{remark}\label{3d}
Let $X$ be a general hypersurface of dimension $3$, 
and $C \subset X$ an ACM local complete intersection curve. 
Assume that condition $(A)$ above is replaced 
by the condition
\begin{itemize}
    \item[(A')] $I_{C/\bbP^4}(e+1)$ is globally generated, and as before that
   \item[(B)] there is a smooth member $S \in \vert I_{Z/X}(e)\vert$ (in particular $e\geq 1$).
\end{itemize} 
Then the above analysis shows that the linear subsystem 
$W \subset \HH^0(\bbP^4, \Oh_{\bbP^4}(e+1))$ given by the exact sequence
$$0 \longrightarrow \HH^0(\bbP^4, I_{C/\bbP^4}(e+1)) \longrightarrow W \longrightarrow \overline{W}_{e+1} \longrightarrow 0$$
is base point free. The same argument as above shows that, by the Beauville-M\'erindol criterion, 
the normal bundle sequence for the inclusions $C \subset X \subset \bbP^4$ splits.

\end{remark}
}



\section{Proofs of Theorem \ref{mainthm1}
and Theorem \ref{ggextendable}}
Before we provide the proofs of Theorem \ref{mainthm1} and Theorem 
\ref{ggextendable}, we will need the following elementary result.
{\color{black}
\begin{lemma}\label{restriction}
Let $X \subset \bbP^{n+1}$ be a smooth hypersurface and 
$Y\subset \bbP^n$ be obtained by taking a smooth hyperplane section. 
Let $E$ be a rank $r$ vector bundle on $X$. Then
\begin{enumerate}
    \item If $E$ is ACM, then so is its restriction $E\tensor\Oh_Y$.
    \item If $E$ is ACM and $E\tensor\Oh_Y$ splits into a sum of line 
    bundles of the form $\Oh_Y(a)$, then so does $E$.
\end{enumerate}
\end{lemma}

\begin{proof}
We have the short exact sequence 
$$0 \longrightarrow E(-1) \longrightarrow E \longrightarrow E\tensor\Oh_Y \longrightarrow 0.$$
Taking cohomology, we have the long exact sequence
$$ \cdots \longrightarrow \HH^i(X, E(j)) \longrightarrow \HH^i(Y,E(j)\tensor\Oh_Y) \longrightarrow 
\HH^{i+1}(X, E(j-1))  \longrightarrow \cdots. $$
Since the extreme terms vanish for $0 < i < n-1$ and $\forall \, j \in \bbZ$, $(1)$  
follows. For $(2)$, write 
$E\tensor\Oh_Y\isom\bigoplus_{j=1}^r\Oh_{Y}(a_j)$ and note 
that the composed map 
\begin{equation*}
    E\longrightarrow E\tensor\Oh_Y\isom\bigoplus_{j=1}^r\Oh_{Y}(a_j)
\end{equation*} lifts to a map 
\begin{equation}\label{composed}
    E\longrightarrow\bigoplus_{j=1}^r\Oh_{X}(a_j)
\end{equation} via the exact sequence 
$$0\longrightarrow\bigoplus_{j=1}^r\Oh_{X}
(a_j-1)\longrightarrow\bigoplus_{j=1}^r\Oh_{X}(a_j)
\longrightarrow\bigoplus_{j=1}^r\Oh_{Y}(a_j)\longrightarrow 0$$ 
as $H^1_*(X,E^{\vee})=0$. Since \eqref{composed} is a map between 
vector bundles of the same rank, we conclude that it is an isomorphism. 
Indeed, this is a consequence of the 
fact that the determinant of the map is non-zero as it is so on $Y$. This 
implies that $E$ is also a direct sum of line bundles.
\end{proof}
}

\begin{proof}[Proof of Theorem \ref{mainthm1}]  
The proof is based by induction on the dimension $n$. Let us first deal with the base case:
{\color{black}
\begin{claim}\label{c2}
Theorem \ref{mainthm1} holds when $n=4$.
\end{claim}
\noindent{\it Proof of Claim \ref{c2}:} Let $Z\subset X$ be a local complete intersection ACM surface satisfying the 
hypotheses of Theorem \ref{mainthm1}. By Lemma \ref{l3}, it is enough to show that $E$ in \eqref{defeq} is a 
direct sum of line bundles. Note that $E$ is ACM and globally generated by Proposition \ref{serre}. {Pick a 
general subspace $V_{r-1}\subset H^0(X,E)$ of dimension $r-1$, which by \cite{Ban} defines a smooth ACM (hence 
irreducible) surface $Z'\subset X$. Moreover we have the exact sequence 
\begin{equation}\label{ft}
    0\longrightarrow\mathcal{O}_X^{\oplus r-1}\longrightarrow E\longrightarrow I_{Z'/X}(e)\longrightarrow 0
\end{equation}
Note that the conditions (ii) and (iii) of Theorem \ref{mainthm2} continue to hold for $Z'\subset X$.}

Now, we note that there is a smooth $\Theta\in |I_{Z'/X}(e)|$ which is obtained by choosing a general $r$-dimensional subspace 
$V_r\subset H^0(X,E)$ containing $V_{r-1}$ and taking $\Theta$ to be the degeneracy locus $D_{r-1}(\phi)$ where 
$$ \phi: V_r\tensor\Oh_X \longrightarrow E.$$
Consequently, the normal bundle sequence for the inclusions $C'\subset Y 
\subset \bbP^{4}$ obtained by taking a general hyperplane section of the 
inclusions $Z' \subset X \subset \bbP^5$ splits by 
Theorem \ref{mainthm2}, whence $E\tensor\Oh_Y$ is a direct sum of line bundles by 
Corollary \ref{mprvcor}. By Lemma \ref{restriction}, $E$ is a direct sum of line bundles as well.
\hfill$\spadesuit$

\smallskip

Let us continue with the proof of Theorem \ref{mainthm1}. 
Now we carry out the induction step. Since the 
assertion is already proven for $n=4$ in Claim \ref{c2}, we assume $n\geq 5$.
Recall the exact sequence
\begin{equation*}
    0 \longrightarrow \Oh_X^{\oplus r-1} \longrightarrow E \longrightarrow I_{Z/X}(e) \longrightarrow 0,
\end{equation*}
where $E$ is a rank $r$ globally generated ACM bundle on $X$ (see Proposition \ref{serre}). Setting $X_n:=X$, $Z_{n-2}:=Z$, 
and repeatedly restricting this sequence by general hyperplane sections $X_i$ of dimension $i$, one obtains 
codimension $2$ subvarieties $Z_{i-2}$ of dimension $i-2$, and the exact sequences
\begin{equation}\label{ex21}
    0 \longrightarrow {\Oh_X}^{\oplus r-1}_i \longrightarrow E_i \longrightarrow I_{Z_{i-2}/X_i}(e) 
    \longrightarrow 0\textrm{ for all }i\geq 4,
\end{equation}
where $E_i:=E|_{X_i}$. By Lemma \ref{restriction}, $E_i$ is ACM for $i\geq 4$, whence $Z_{i-2}\subset X_i$ is ACM 
by \eqref{ex21} for $i$ in the same range. As a result, the pair $(Z_{i-2},X_i)$ satisfies the hypotheses of 
the Theorem for all $i\geq 4$. 

Now, $E_4$ is a direct sum of line bundles by the proof of Claim \ref{c2}, 
and by Lemma \ref{restriction}, so is $E$. 
Hence $Z$ is extendable by Corollary \ref{mprvcor}.

{\color{black} 
When $n=3$, the proof is as in the proof of the $n=4$ case, i.e., Claim 
\ref{c2}, except that we invoke Remark \ref{3d} instead of Theorem 
\ref{mainthm2}.
}
}
\end{proof}

\begin{proof}[Proof of Theorem \ref{ggextendable}]
{\color{black}
By the arguments in the proof of Theorem \ref{mainthm1}, it is enough to show that $E_4$ is a direct sum of line 
bundles where $E_4:=E|_{X_4}$ and $X_4$ is a obtained by intersecting $n-4$ general hyperplane sections of $X$. 
As $E_4$ is globally generated with $c_1(E_4)=\Oh_{X_4}(e)$, a choice of $r-1$ general sections yields an exact sequence
$$0 \longrightarrow \Oh_{X_4}^{\oplus r-1} \longrightarrow E_4 \longrightarrow I_{Z_2/X_4}(e) \longrightarrow 0,$$
where $I_{Z_2/X_4}$ is the ideal sheaf of a pure codimension $2$ smooth ACM subscheme $Z_2$ in $X_4$. First assume 
$Z_2=\emptyset$ whence $I_{Z_2/X_4}=\Oh_{X_4}$. In this case, clearly the above exact sequence is split as 
$\HH^1_*(X_4,\Oh_{X_4})=0$ whence $E_4$ is a direct sum of line bundles. So, we may assume $Z_2\neq\emptyset$, in 
particular $H^0(X_4,I_{Z_2/X_4})=0$. Since $Z_2\subset X_4$ is ACM, we see that $h^0(Z_2,\Oh_{Z_2})=1$, 
in particular $Z_2$ is irreducible. Therefore, applying the proof of Claim \ref{c2}, we conclude that 
$E_4$ is a direct sum of line bundles. 
}
\end{proof}


\begin{thebibliography}{BKKMSU}

 \bibitem[Arr07]{arrondo} Arrondo, E., {\it A home-made Hartshorne-Serre correspondence},
Rev. Mat. Complut. 20 (2007), no. 2, 423--443. 

 \bibitem[Ban91]{Ban} Banica, C., {\it Smooth reflexive sheaves},
Rev. Roumaine Math. Pures Appl. 36 (1991), no. 9-10, 571-593.

 \bibitem[Bea00]{Bea'} Beauville, A., {\it Determinantal hypersurfaces},
Michigan Math. J. 48 (2000), 39-64.

 \bibitem[Bea18]{Bea} Beauville, A., {\it An introduction to Ulrich bundles}, Eur. J. Math. 4 (2018), no. 1, 26–36.

\bibitem[BGS87]{BGS} Buchweitz, R.-O., Greuel, G.-M., Schreyer, F.-O., 
{\it Cohen–Macaulay modules on hypersurface singularities II},
Invent. Math. 88 (1987) 165–182.

\bibitem[BM87]{BM87} Beauville, A., Mérindol, J.-Y., {\it Sections hyperplanes des surfaces K3},
Duke Math. J. 55 (1987), no. 4, 873--878. 

\bibitem[BMK13]{BK}  Beheshti, R., Mohan Kumar, N., 
{\it  Spaces of rational curves on complete intersections}, 
Compos. Math. 149 (2013), no. 6, 1041--1060.

\bibitem[BR22]{BR}  Beheshti, R., Riedl, E., 
{\it  Restrictions on rational surfaces lying in very general hypersurfaces}, 
Forum Math. Sigma 10 (2022), Paper No. e71, 11 pp. 

\bibitem[CMR$^+$21]{CMRPL} Costa, L., Mir\'o-Roig, R. M., Pons-Llopis J., {\it  Ulrich bundles}, De Gruyter Studies in Mathematics, 77, De Gruyter 2021.

\bibitem[Cos17]{EC17} Coskun, E., {\it A survey of Ulrich bundles}, Analytic and algebraic geometry, 
85--106, Hindustan Book Agency, New Delhi, 2017.

\bibitem[ES03]{ES} Eisenbud, D., Schreyer, F.-O. (with an appendix by Weyman, J.), 
{\it  Resultants and Chow forms via exterior syzygies.}, J. Amer. Math. Soc. 16 (2003), no. 3, 537–579.

\bibitem[Fae13]{Faenzi} Faenzi, Daniele, {\it Some applications of vector bundles in algebraic geometry},
Habilitation \`a diriger des recherches, 2013.

\bibitem[GH83]{GH1} Griffiths, P., Harris, J., {\it Infinitesimal variations of Hodge structure. II. An infinitesimal invariant of
Hodge classes}, Compos. Math. 50 (1983), no. 2-3, 207--265.

\bibitem[GH85]{GH} Griffiths, P., Harris, J., {\it On the Noether-Lefschetz theorem and some remarks on codimension
two cycles}, Math. Ann. 271 (1985), no. 1, 31--51.

\bibitem[Gre88]{Green} Green, Mark L., {\it  A new proof of the explicit Noether-Lefschetz theorem}, 
J. Differential Geom. 27 (1988), no. 1, 155--159.

\bibitem[IN02]{IN} Inamdar, S. P., Nagaraj, D. S., {\it Cycle class map and restriction of subvarieties},
J. Ramanujan Math. Soc. 17 (2002), no. 2, 85--91. 

\bibitem[Lop23]{Lop}  Lopez, A. F. (with an appendix by Thomas Dedieu), {\it On the extendability of projective varieties: a survey}, Trends Math. Birkhäuser/Springer, Cham, 2023, 261–291.

\bibitem[LR24a]{LR1}  Lopez, A. F., Raychaudhury, D., {\it Ulrich subvarieties and the non-existence of low rank Ulrich bundles on complete intersections}, arXiv:2405.01154.

\bibitem[LR24b]{LR2}  Lopez, A. F., Raychaudhury, D., {\it Non-existence of low rank Ulrich bundles on Veronese varieties}, arXiv:2406.08162.

\bibitem[MKRR07]{MPR2}  Mohan Kumar, N., Rao, A.P.,  Ravindra, G.V., {\it Arithmetically Cohen-Macaulay bundles on three dimensional
hypersurfaces}, Int. Math. Res. Not. IMRN (2007) no. 8, Art. ID rnm025, 11 pp.

\bibitem[MKRR09]{MPRV} Mohan Kumar, N., Rao, A.P.,  Ravindra, G.V., {\it On codimension two subvarieties of hypersurfaces}, 
Motives and algebraic cycles,  167--174, Fields Inst. Commun., 56, Amer. Math. Soc., Providence, RI, 2009.

\bibitem[Ott95]{Ott} Ottaviani, G., {\it Variet\`a proiettive di codimensione piccola}, 
Quaderni INdAM. Aracne, 1995.

\bibitem[Pac04]{Pa} Pacienza, G., {\it Subvarieties of general type on a general projective hypersurface},
Trans. Amer. Math. Soc., vol. 356, no. 7, (2004), 2649--2661.

\bibitem[Rav09]{R} Ravindra, G.V., {\it Curves on threefolds and a conjecture of Griffiths-Harris,} Math Ann.  345  (2009),  no. 3, 731--748.

\bibitem[RT19]{RT4} Ravindra, G.V., Tripathi, A., {\it Rank 3 ACM bundles on general hypersurfaces in 
$\bbP^5$}, Adv. Math. 355 (2019), 106780, 33 pp.

\bibitem[RT22]{RT6} Ravindra, G.V., Tripathi, A., {\it On the base case of a conjecture on ACM bundles over hypersurfaces},
Geom. Dedicata 216 (2022), issue 5, 10pp.

\bibitem[Tri16]{AT1} Tripathi, A., {\it Splitting of low-rank ACM bundles on hypersurfaces of high dimension}, 
Comm. Algebra 44 (2016), no. 3, 1011--1017.

\bibitem[Tri17]{AT2} Tripathi, A., {\it  Rank 3 arithmetically Cohen-Macaulay bundles over hypersurfaces},
J. Algebra 478 (2017), 1--11. 

\bibitem[Voi88]{Voisin} Voisin, C., {\it Sur une conjecture de Griffiths et Harris}, 
Algebraic curves and projective geometry,
(Trento, 1988), 270--275, Lecture Notes in Math., 1389, Springer, Berlin, 1989.

\bibitem[Voi92]{V} Voisin, C., {\it Sur l’application de Wahl des courbes satisfaisant la condition de Brill-Noether-Petri},
Acta Math. 168 (1992), no. 3-4, 249-272. 

\bibitem[Voi96]{VoisinJDG} Voisin, C., {\it On a conjecture of Clemens on rational curves on hypersurfaces},
J. Differential Geom. 44 (1996), no. 1, 200--213. 

\bibitem[Wah87]{Wahl87} Wahl, Jonathan M., {\it The Jacobian algebra of a graded Gorenstein singularity},
Duke Math. J. 55 (1987), no. 4, 843--871. 

\end{thebibliography}
\end{document}